\documentclass[reqno,12pt]{amsart}
\usepackage{amsmath,amssymb,color}
\usepackage{amsthm}
\usepackage{graphicx}

\newtheorem{theorem}{Theorem}

\newtheorem{remark}{Remark}
\newtheorem{lemma}{Lemma}

\def\re{\mathbb{R}}

\def\eps{\varepsilon}

\def\pd{\partial}
\def\ol{\overline}

\def\la{\lambda}

\def\({\left(}
\def\){\right)}

\def\pd{\partial}

\def\|{\Vert}

\begin{document}
\title[One-dimensional boundary blow up problem]{One-dimensional boundary blow up problem with a nonlocal term}

\author{Taketo Inaba}

\address{
Department of Mathematics, Osaka Metropolitan University \\
3-3-138, Sumiyoshi-ku, Sugimoto-cho, Osaka, Japan \\
}

\email{sf22812v@st.omu.ac.jp \\}

\author{Futoshi Takahashi}

\address{
Department of Mathematics, Osaka Metropolitan University \\
3-3-138, Sumiyoshi-ku, Sugimoto-cho, Osaka, Japan \\
}

\email{futoshi@omu.ac.jp \\}

\begin{abstract}
In this paper, we study a nonlocal boundary blow up problem on an interval and obtain the precise asymptotic formula for solutions 
when the bifurcation parameter in the problem is large.
\end{abstract}


\date{\today}

\dedicatory{}

\maketitle

\section{Introduction}

In this paper, we study the following one-dimensional nonlocal elliptic problem 
\begin{equation}
\label{BBP_M}
	\begin{cases}
	M(\| u \|_{L^q(I)}) u''(x) = \la u^p(x), \quad x \in I = (-1,1), \\
	u(x) > 0, \quad x \in I, \\
	u(x) \to +\infty  \quad \text{as} \ x \to \pd I = \{-1, +1\},
	\end{cases}
\end{equation}
where $\la > 0$ is a given constant and $M: [0,+\infty) \to \re_+$ is a continuous function. 
We call $u$ a solution of \eqref{BBP_M} if $u \in C^2(I) \cap L^q(I)$ and $u$ solves the equation for all $x \in I$.
Throughout of the paper, we assume that
\begin{equation}
\label{Assumption}
	p > 1 \quad \text{and} \quad 0 < q < \frac{p-1}{2}.
\end{equation}
We consider $\la > 0$ in \eqref{BBP_M} as a bifurcation parameter;
the number of solutions of \eqref{BBP_M} is changed according to the value of $\la$.
In most cases we consider $M(t) = (t^q + b)^r$ $(t \ge 0)$ for constants $r>0$ and $b \ge 0$, 
i.e., we consider
\begin{equation}
\label{BBP}
	\begin{cases}
	\(\| u \|_{L^q(I)}^q + b \)^r u''(x) = \la u^p(x), \quad x \in I = (-1,1), \\
	u(x) > 0, \quad x \in I, \\
	u(x) \to +\infty  \quad \text{as} \ x \to \pd I =  \{-1, +1\}.
	\end{cases}
\end{equation}

For any $p > 1$, it is a classical fact that there exists a unique solution $U_p \in C^2(I)$ satisfying
\begin{equation}
\label{Eq:U_p}
	\begin{cases}
	U_p''(x) = U_p^p(x), \quad x \in I = (-1,1), \\
	U_p(x) > 0, \quad x \in I, \\
	U_p(x) \to +\infty  \quad \text{as} \ x \to \pd I =  \{-1, +1\},
	\end{cases}
\end{equation}
see Proposition 1.8 and Remark 1.10 in \cite{DDGR} (for example).
Note that the nonlinearity $f(s) = s^p$ $(p > 1)$ satisfies the famous Keller-Osserman condition \cite{Keller}, \cite{Osserman}
\[
	\int_a^{\infty} \frac{1}{\sqrt{F(t)}} dt < +\infty \quad \text{for some} \quad a > 0,
\]
where $F(t) = \int_0^t f(s) ds$. 
By the uniqueness, we further see that $U_p$ is even: $U_p(-x) = U_p(x)$ for any $x \in I$;
otherwise $V_p(x) = U_p(-x)$ will be another solution of \eqref{Eq:U_p} different from $U_p$.
Later in Lemma \ref{Lemma:U_p}, we prove that $U_p \in L^q(I)$ for $0 < q < \frac{p-1}{2}$ and we compute the exact value of $\| U_p \|_{L^q(I)}$.

In the following, we put 
\begin{equation}
\label{Def:L_p}
	L_p = \int_{1}^{\infty}\frac{dt}{\sqrt{t^{p+1}-1}}
\end{equation}
for $p > 1$. 
Also $B(x,y) = \int_0^1 t^{x-1}(1-t)^{y-1} dt$ denotes the Beta function.

First, we consider the case $b=0$ in \eqref{BBP}.

\begin{theorem}
\label{Thm:b=0}
Assume \eqref{Assumption} and let $r > 0$ and $b = 0$ in \eqref{BBP}.
Denote $B_{p,q} =  B\(\frac{p-2q-1}{2(p+1)}, \frac{1}{2}\)$.

\begin{enumerate}
\item[(i)] Assume $qr-p+1 \ne 0$.
Then for any $\la > 0$, there exists a unique solution $u_{\la} \in C^2(I) \cap L^q(I)$ of \eqref{BBP} which satisfies 
\begin{align}
\label{Formula:b=0}
	&u_{\la}(x) = \la^{\frac{1}{qr-p+1}} \| U_p \|_q^{-\frac{qr}{qr-p+1}} U_p(x) \notag \\
	&= \la^{\frac{1}{qr-p+1}} \( \frac{p+1}{2} \)^{\frac{r(p-1-q)}{(p-1)(qr-p+1)}} L_p^{-\frac{r(2q-p+1)}{(p-1)(qr-p+1)}} B_{p,q}^{-\frac{r}{qr-p+1}} U_p(x)
\end{align}
for $x \in I$, where $U_p$ is the unique solution of \eqref{Eq:U_p}. 

\item[(ii)] Assume $qr-p+1= 0$.
Then \eqref{BBP} admits a solution if and only if $\la =  \| U_p \|_{L_q(I)}^{qr} = \( \frac{p+1}{2} \)^{-\frac{r(p-1-q)}{p-1}} L_p^{\frac{r(2q-p+1)}{p-1}} B_{p,q}^r$.
In this case, $u_{\la} \Big |_{\la = \| U_p \|_q^{qr}} = U_p$. 
\end{enumerate}
\end{theorem}

Next, we consider the case $b>0$ in \eqref{BBP}.
For the following theorems, we define
\begin{equation}
\label{Def:g(t)}
	g(t) = \frac{(t^q + b)^r}{t^{p-1}} \quad \text{for} \ t > 0.
\end{equation}

\begin{theorem}
\label{Thm:b>0(I)}
Assume \eqref{Assumption} and let $r > 0$ and $b > 0$ in \eqref{BBP}.
Assume $qr-p+1 > 0$ and put
\begin{align*}
	t_0 = \(\frac{b(p-1)}{qr-p+1}\)^{\frac{1}{q}}, \quad \la_0 = \| U_p \|_q^{p-1} g(t_0)
\end{align*}
where $g$ is in \eqref{Def:g(t)}.
Then
\begin{enumerate}
\item[(i)] If $0 < \la < \la_0$, there exists no solution in $C^2(I) \cap L^q(I)$ for \eqref{BBP}. 
\item[(ii)] If $\la = \la_0$, there exists a unique solution in $C^2(I) \cap L^q(I)$ for \eqref{BBP}. 
\item[(iii)] If $\la > \la_0$, there exist exactly two solutions $u_{1,\la}$ and $u_{2,\la}$ in $C^2(I) \cap L^q(I)$ for \eqref{BBP}. 
Moreover, as $\la \to +\infty$, 
\begin{align*}
	&u_{1,\la}(x) = b^{\frac{r}{p-1}} \la^{-\frac{1}{p-1}} \(1 + \frac{r}{p-1} b^{\frac{qr-p+1}{p-1}} \la^{-\frac{q}{p-1}} \| U_p \|_q^q(1 + o(1)) \) U_p(x), \\
	&u_{2,\la}(x) = \( m_{p,q} \la^{\frac{1}{qr-p+1}} - \frac{brm_{p,q}^{1-q}}{qr-p+1} \la^{\frac{1-q}{qr-p+1}} (1 + o(1)) \) \frac{U_p(x)}{\| U_p \|_q}
\end{align*}
for any $x \in I$, where $m_{p,q} = \| U_p \|_q^{\frac{1-p}{qr-p+1}}$.
\end{enumerate}
\end{theorem}

\begin{theorem}
\label{Thm:b>0(II)}
Assume \eqref{Assumption} and let $r > 0$ and $b > 0$ in \eqref{BBP}.
Assume $qr-p+1 \le 0$.
Then
\begin{enumerate}
\item[(i)] If $qr-p+1<0$, there exists a unique solution in $C^2(I) \cap L^q(I)$ of \eqref{BBP} for any $\la > 0$. 
\item[(ii)] If $qr-p+1=0$, there exists a unique solution in $C^2(I) \cap L^q(I)$ of \eqref{BBP} for any $\la > \| U_p \|_q^{p-1}$
and no solution for $\la \le \| U_p \|_q^{p-1}$.
\end{enumerate}
\end{theorem}

As a special case of Theorem \ref{Thm:b>0(II)}, we have the following.

\begin{theorem}
\label{Thm:b>0(III)}
Let $p>1$, $b > 0$, and  $q = 1 < \frac{p-1}{2} = r$.
Then for any $\la > 0$, there exists a unique solution $u_{\la} \in C^2(I) \cap L^q(I)$ of \eqref{BBP} of the form
\[
	u_{\la}(x) = \frac{\| U_p\|_1 + \sqrt{\| U_p\|_1^2 + 4b \la^{\frac{2}{p-1}}}}{2} \la^{-\frac{2}{p-1}} U_p(x).
\] 
\end{theorem}

\vspace{1em}
The problem \eqref{BBP} is a one-dimensional toy model of much more general boundary blow up problem
with a nonlocal term
\begin{equation}
\label{BBP_f}
	\begin{cases}
	&M(\| u \|_{L^q(\Omega)}) \Delta u = \la f(u) \quad \text{in} \ \Omega, \\
	&u > 0 \quad \text{in} \ \Omega, \\
	&u(x) \to +\infty \quad \text{as} \ d(x) := {\rm dist }(x, \pd\Omega) \to 0,
	\end{cases}
\end{equation}
where $\Omega$ is a bounded domain in $\re^N$, $N \ge 1$, $f$ is a continuous nonlinearity, $\la > 0$, $q > 0$,
and $M:[0,\infty) \to \re_+$ is a continuous function.
$u$ is called a solution of \eqref{BBP_f} if $u \in C^2(\Omega) \cap L^q(\Omega)$ and $u$ satisfies the equation for $x \in \Omega$.
See some discussion for $f(u) = u^p$ in \S 6.
As far as the authors know, the problem \eqref{BBP_f}, or even \eqref{BBP_M}, \eqref{BBP}, have not been studied so far.
Therefore in this paper, we start to study \eqref{BBP} and obtain the precise information of the solution set 
according to the bifurcation parameter $\la$.

\section{Exact $L^q$-norm of $U_p$}

In this section, first we compute the precise value of the $L^q$-norm of $U_p$, which is the unique solution of \eqref{Eq:U_p},
when $p$ and $q$ satisfies \eqref{Assumption}.

\begin{lemma}
\label{Lemma:U_p}
For a given $p > 1$, let $U_p$ be the unique solution of the problem \eqref{Eq:U_p}.
Then $U_p \in L^q(I)$ if and only if $0 < q < \frac{p-1}{2}$.
Also we obtain
\begin{align}
\label{mu_p}
	&\mu_p := \min_{x \in I} U_p(x) = U_p(0) = \( \sqrt{\frac{p+1}{2}} L_p \)^{\frac{2}{p-1}}, \\ 
\label{L^qU_p}
	&\| U_p \|_{L_q(I)}^q = \sqrt{\frac{2}{p+1}}  \mu_p^{\frac{2q-p+1}{2}}  B\(\frac{p-2q-1}{2(p+1)},\frac{1}{2}\),
\end{align}
where $L_p$ is defined in \eqref{Def:L_p} and $B(x,y)$ denotes the Beta function.
\end{lemma}

\begin{proof}
Following the paper by Shibata \cite{Shibata(JMAA)}, we use a time-map method to compute $\| U_p \|_{L^q(I)}$.
In \cite{Shibata(JMAA)}, the unique solution $W_p$ of the problem
\begin{equation*}
	\begin{cases}
	-W_p''(x) = W_p^p(x), \quad x \in I = (-1,1), \\
	W_p(x) > 0, \quad x \in I = (-1,1), \\
	W_p(\pm 1) = 0
	\end{cases}
\end{equation*}
is consdered and its maximum value and the $L^p$-norm is computed.
 
Multiply the equation by $U_p'(x)$, we have
\[
	(U_p''(x)-U_p^p(x))U_p'(x) = 0.
\]
This implies
\begin{align*}
	\left\{\frac{1}{2}(U_p'(x))^{2}-\frac{1}{p+1}U_p^{p+1}(x)\right\}' = 0,
\end{align*}
thus 
\begin{align*}
	\frac{1}{2}(U_p'(x))^{2}-\frac{1}{p+1}U_p^{p+1}(x) = \text{constant} = 0-\frac{1}{p+1} \mu_p^{p+1}.
\end{align*}
Therefore, we have
\begin{equation}
\label{Uprime}
	U_p'(x) = \sqrt{\frac{2}{p+1}(U_p^{p+1}(x)-\mu_p^{p+1})}.
\end{equation}
By \eqref{Uprime} and the change of variables $s = U_p(x)$ and $s= \mu_p t$, we compute
\begin{align*}
	1 = \int_{0}^{1}1dx &= \int_{0}^{1}\frac{U_p'(x)dx}{\sqrt{\frac{2}{p+1}(U_p^{p+1}(x)-\mu_p^{p+1})}} \\
	&=\sqrt{\frac{p+1}{2}}\int_{\mu_p}^{\infty}\frac{ds}{\sqrt{(s^{p+1}-\mu^{p+1})}} \\
	&=\sqrt{\frac{p+1}{2}}\int_{1}^{\infty} \frac{\mu_p dt}{\sqrt{\mu_p^{p+1}(t^{p+1}-1)}} \\
	&=\sqrt{\frac{p+1}{2}}\mu_p^{\frac{1-p}{2}} \int_{1}^{\infty}\frac{dt}{\sqrt{t^{p+1}-1}}.
\end{align*}
From this, we obtain
\[
	\mu_p = U_p(0) = \( \sqrt{\frac{p+1}{2}} L_p \)^{\frac{2}{p-1}}.
\]
Next, by the formula $\mu_p$, we see
\begin{align*}
	\| U_p \|_{L_q(I)}^{q} &= 2\int_0^1 |U_p(x)|^q dx \\
	&=2\int_{0}^{1}U_p^{q}(x)\frac{U_p'(x)dx}{\sqrt{\frac{2}{p+1}(U_p^{p+1}(x)-\mu_p^{p+1})}}\\
	&=\sqrt{2(p+1)}\int_{\mu_p}^{\infty}\frac{s^q ds}{\sqrt{(s^{p+1}-\mu_p^{p+1})}} \quad (s=U_p(x))\\
	&=\sqrt{2(p+1)}\int_{1}^{\infty} \frac{\mu^{q+1}t^{q} dt}{\sqrt{\mu_p^{p+1}(t^{p+1}-1)}} \quad \left(t=\frac{s}{\mu_p}\right)\\
	&=\sqrt{2(p+1)}\mu_p^{\frac{2q-p+1}{2}}\int_{1}^{\infty}\frac{t^q dt}{\sqrt{t^{p+1}-1}}.
\end{align*}
Since
\begin{align*}
	\int_{1}^{\infty}\frac{t^q dt}{(t^{p+1}-1)^{\frac{1}{2}}} &=\int_0^1 \frac{s^{\frac{-p-2q-3}{2(p+1)}}}{\sqrt{1-s}}ds \quad \(s=\frac{1}{t^{p+1}} \)\\
	&=\frac{1}{p+1}B \(\frac{p-2q-1}{2(p+1)},\frac{1}{2} \),
\end{align*}
we obtain the formula of $\| U_p \|_{L^q(I)}$.
\end{proof}

\begin{remark}
Since $L_p$ is uniformly bounded with respect to $p$, we see
\begin{align*}
	\log \mu_p &=\frac{2}{p-1}\sqrt{\frac{p+1}{2}}+\frac{2}{p-1}\log L_p \to 0 \quad (p \to \infty).
\end{align*}
This implies $\lim_{p \to \infty} \mu_p = 1$.
\end{remark}

\begin{remark}
By the homogeneity of the nonlinearity $f(s) = s^p$ and the uniqueness of $U_p$, 
it is obvious that for any $\la > 0$, $w_{\la}(x) = \la^{-\frac{1}{p-1}} U_p(x)$ is the unique solution of the problem
\begin{equation}
\label{Eq:w}
	\begin{cases}
	w''(x) = \la w^p(x), \quad x \in I = (-1,1), \\
	w(x) > 0, \quad x \in I, \\
	w(\pm 1) = +\infty.
	\end{cases}
\end{equation}
\end{remark}

Next lemma is an analogue of Theorem 2 in \cite{ACM} and is important throughout of the paper. 

\begin{lemma}
\label{Lemma:M}
Let $M:[0,+\infty) \to \re_+$ be continuous in \eqref{BBP_M}.
Then for a given $\la > 0$, the problem \eqref{BBP_M} has the same number of positive solutions as that of the positive solutions 
(with respect to $t >0$) of
\begin{equation}
\label{Eq:M}
	\frac{M(t)}{t^{p-1}} = \la \| U_p \|_q^{1-p}
\end{equation}
where $U_p$ is the unique solution of \eqref{Eq:U_p}.
Moreover, any solution $u_{\la}$ of \eqref{BBP_M} must be of the form 
\begin{equation}
\label{u_form}
	u_{\la}(x) = t_{\la} \frac{U_p(x)}{\| U_p \|_q}
\end{equation}
where $t_{\la}>0$ is a solution of \eqref{Eq:M}.
\end{lemma}

\begin{proof}
Let $u \in C^2(I) \cap L^q(I)$ be any solution of \eqref{BBP_M} and put $v = \gamma u$, 
where 
\[
	\gamma = M(\| u \|_q)^{-\frac{1}{p-1}} \la^{\frac{1}{p-1}}.
\]
Then
\begin{align*}
	v''(x) &= \gamma u''(x) \overset{\eqref{BBP_M}}{=} \gamma \frac{\la}{M(\| u \|_q)} u^p(x) \\
	&= \gamma \frac{\la}{M(\| u \|_q)} \(\frac{v(x)}{\gamma}\)^p \\
	&= \(\gamma^{1-p} \frac{\la}{M(\| u \|_q)} \) v^p(x) \\
	&= v^p(x).
\end{align*}
Thus $v$ is a solution of 
\[
	\begin{cases}
	v''(x) = v^p(x), \quad x \in I = (-1,1), \\
	v(x) > 0, \quad x \in I, \\
	v(\pm 1) = +\infty.
	\end{cases}
\]
By the uniqueness, $v \equiv U_p$.
This leads to 
\begin{equation}
\label{u_form2}
	u = \gamma^{-1} U_p = M(\| u \|_q)^{\frac{1}{p-1}} \la^{-\frac{1}{p-1}} U_p.
\end{equation}
Now, we put $t = \| u \|_q > 0$. Then by taking a $L^q$-norm of the both sides of \eqref{u_form2}, we have
\[
	t = \| u \|_q = M(t)^{\frac{1}{p-1}} \la^{-\frac{1}{p-1}} \| U_p \|_q,
\]
which is equivalent to \eqref{Eq:M}.
This shows that
\[
	\sharp \{ u \in C^2(I) \cap L^q(I): \text{solutions of \eqref{BBP_M}} \} \le \sharp \{ t > 0: \text{solutions of \eqref{Eq:M}} \}
\]
where $\sharp A$ denotes the cardinality of the set $A$.

On the other hand, let $t > 0$ be any solution of \eqref{Eq:M} and put 
\[
	u(x) = t \frac{U_p(x)}{\| U_p \|_q} \in C^2(I) \cap L^q(I).
\]
Then we have $u(x) > 0$, $u(\pm 1) = +\infty$, $t = \| u \|_q$ and
\begin{align*}
	M(\| u \|_q) u''(x) &= M(t) \frac{t}{\| U_p \|_q} U_p''(x) \overset{\eqref{Eq:U_p}}{=} M(t) \frac{t}{\| U_p \|_q} U_p^p(x) \\
	&= M(t) \frac{t}{\| U_p \|_q} \(\frac{\| U_p \|_q}{t} u(x) \)^p \\
	&= M(t) \frac{\| U_p \|_q^{p-1}}{t^{p-1}} u^p(x) \\
	&= \la u^p(x).
\end{align*}
This shows that
\[
	\sharp \{ t > 0: \text{solutions of \eqref{Eq:M}} \} \le \sharp \{ u \in C^2(I) \cap L^q(I): \text{solutions of \eqref{BBP_M}} \}.
\]
\end{proof}

\section{Proof of Theorem \ref{Thm:b=0}.}

In this section, we prove Theorem \ref{Thm:b=0}.

\begin{proof}
Let $qr+1-p \ne 0$.
By virtue of Lemma \ref{Lemma:M}, we only need to prove that \eqref{Eq:M} has the unique solution for any $\la > 0$.
Since $b=0$ in \eqref{BBP}, $M(t) = t^{qr}$ and the equation \eqref{Eq:M} reads
\begin{equation}
\label{Eq:b=0}
	t^{qr-p+1} = \la \| U_p \|_q^{1-p}.
\end{equation}
The map $t \mapsto t^{qr-p+1}$ is strictly monotone increasing (resp. decreasing) for $t >0$ 
if $qr-p+1 >0$ (resp. if $qr-p+1 <0$) and  maps $(0, +\infty)$ onto $(0, +\infty)$.
Therefore, the equation \eqref{Eq:b=0} admits the unique solution
\[
	t_{\la} = \la^{\frac{1}{qr-p+1}} \| U_p \|_q^{\frac{1-p}{qr-p+1}}
\]
for any $\la >0$.
By \eqref{u_form}, this corresponds to the unique solution
\[
	u_{\la}(x) = t_{\la} \frac{U_p(x)}{\| U_p \|_q} = \la^{\frac{1}{qr-p+1}} \| U_p \|_q^{\frac{-qr}{qr-p+1}} U_p(x)
\]
of \eqref{BBP}.
Inserting the formula \eqref{L^qU_p} in $\| U_p \|_q$, we obtain \eqref{Formula:b=0}.
This proves (i).

If $qr-p+1 = 0$, \eqref{Eq:b=0} is satisfied if and only if $\la = \| U_p \|_q^{p-1} = \| U_p \|_q^{qr}$.
Thus in this case $u_{\la} = U_p$ follows. This proves (ii).
\end{proof}

\section{Proof of Theorem \ref{Thm:b>0(I)}.}

In this section, we prove Theorem \ref{Thm:b>0(I)}.

\begin{proof}
According to Lemma \ref{Lemma:M}, the number of solutions of \eqref{BBP} is the same as the number of solutions
\begin{equation}
\label{Eq:g(t)}
	g(t) = \frac{(t^q + b)^r}{t^{p-1}} = \la \| U_p \|_q^{1-p}.
\end{equation}
Since
\[
	g'(t) = \frac{(t^q+b)^{r-1}}{t^p} \left\{ (qr-p+1)t^q - b(p-1) \right\}, 
\]
and $qr-p+1 > 0$,the equation $g'(t) = 0$ admits the unique solution 
\[
	t_0 = \( \frac{b(p-1)}{qr-p+1} \)^{1/q}
\]
and $g'(t) < 0$ if $0<t<t_0$, $g'(t) > 0$ if $t>t_0$ and $g(t_0) = \min_{t \in \re_+} g(t)$. 
Also $\lim_{t \to +0} g(t) = +\infty$ since $b > 0$, and $\lim_{t + \infty} g(t) = +\infty$ since $qr>p-1$.
Thus the equation (in $t$) \eqref{Eq:g(t)} admits 
\begin{enumerate}
\item[(i)] no solution if $\la \| U_p \|_q^{1-p} < g(t_0)$.
\item[(ii)] exactly one solution if $\la \| U_p \|_q^{1-p} = g(t_0)$.
\item[(iii)] exactly two solutions if $\la \| U_p \|_q^{1-p} > g(t_0)$.
\end{enumerate}
Thus putting $\la_0 = \| U_p \|_q^{p-1} g(t_0)$ and recalling Lemma \ref{Lemma:M},
we conclude the former part of Theorem \ref{Thm:b>0(I)}. 

Now, we prove the asymptotic formulae in the case (iii) and $\la >>1$.
By a simple consideration using the graph of $g(t)$, 
we see that two solutions $0 < t_1 < t_2$ of \eqref{Eq:g(t)} satisfy $t_1 < t_0 < t_2$ and
\[
	t_1 \to +0 \quad \text{and} \quad t_2 \to +\infty
\]
as $\la \to +\infty$.
Since
\[
	\la \| U_p \|_q^{1-p} = g(t_2) = \frac{(t_2+b)^r}{t_2^{p-1}} \sim \frac{t_2^{qr}}{t_2^{p-1}}
\]
as $\la \to +\infty$, $t_2$ can be expressed as
\[
	t_2 = \la^{\frac{1}{qr-p+1}} \| U_p \|_q^{\frac{1-p}{qr-p+1}} + R
\]
where $R = o (\la^{\frac{1}{qr-p+1}})$ as $\la \to \infty$.
In the following, we put
\[
	m_{p,q} = \| U_p \|_q^{\frac{1-p}{qr-p+1}}
\]
for simplicity. Then we write 
\begin{equation}
\label{t_2asymp}
	t_2 = m_{p,q} \la^{\frac{1}{qr-p+1}} + R.
\end{equation}
We insert this expression into $g(t_2) = \la \| U_p \|_q^{1-p}$, which is equivalent to
\begin{equation}
\label{Eq:t_2}
	(t_2^q + b)^r = t_2^{p-1} \la \| U_p \|_q^{1-p}.
\end{equation}
Then Taylor expansion implies
\begin{align*}
	&\text{(LHS) of \eqref{Eq:t_2}} = t_2^{qr} \( 1 + \frac{b}{t_2^q} \)^r
	= t_2^{qr} \( 1 + r \frac{b}{t_2^q} (1 + o(1)) \) \\
	&\overset{\eqref{t_2asymp}}{=}  \( m_{p,q} \la^{\frac{1}{qr-p+1}} + R \)^{qr} \( 1 + \frac{rb}{\(m_{p,q} \la^{\frac{1}{qr-p+1}} + R\)^q} (1 + o(1)) \) \\
	&=  \( m_{p,q} \la^{\frac{1}{qr-p+1}} \)^{qr} \(1 + \frac{R}{m_{p,q} \la^{\frac{1}{qr-p+1}}} \)^{qr} \( 1 + \frac{rb}{\(m_{p,q} \la^{\frac{1}{qr-p+1}} \)^q} (1 + o(1)) \) \\
	&=  \( m_{p,q} \la^{\frac{1}{qr-p+1}} \)^{qr} \(1 + \frac{qrR}{m_{p,q} \la^{\frac{1}{qr-p+1}}} (1 + o(1) \) \( 1 + \frac{rb}{\(m_{p,q} \la^{\frac{1}{qr-p+1}} \)^q} (1 + o(1)) \)  \\
	&=  \( m_{p,q} \la^{\frac{1}{qr-p+1}} \)^{qr} \(1 + \( \frac{qrR}{m_{p,q} \la^{\frac{1}{qr-p+1}}} + \frac{rb}{\(m_{p,q} \la^{\frac{1}{qr-p+1}} \)^q} \) (1 + o(1)) \) 
\end{align*}
as $\la \to \infty$.
On the other hand, 
\begin{align*}
	&\text{(RHS) of \eqref{Eq:t_2}} = t_2^{p-1} \la \| U_p \|_q^{1-p} 
	\overset{\eqref{t_2asymp}}{=} \la \| U_p \|_q^{1-p} \(m_{p,q} \la^{\frac{1}{qr-p+1}} + R \)^{p-1} \\
	&= \la \| U_p \|_q^{1-p} m_{p,q}^{p-1} \la^{\frac{p-1}{qr-p+1}} \(1 + (p-1) \frac{R}{m_{p,q} \la^{\frac{1}{qr-p+1}}} \)(1 + o(1)) \\
	&= \( m_{p,q} \la^{\frac{1}{qr-p+1}} \)^{qr} \(1 + (p-1) \frac{R}{m_{p,q} \la^{\frac{1}{qr-p+1}}} \)(1 + o(1))
\end{align*}
as $\la \to \infty$.
We compare these two equations and obtain
\[
	\frac{qrR}{m_{p,q} \la^{\frac{1}{qr-p+1}}} + \frac{br}{(m_{p,q} \la^{\frac{1}{qr-p+1}})^q} = (p-1) \frac{R}{m_{p,q} \la^{\frac{1}{qr-p+1}}} (1 + o(1)).
\]
From this, we obtain
\[
	R = -\frac{brm_{p,q}^{1-q}}{qr-p+1} \la^{\frac{1-q}{qr-p+1}} (1 + o(1))
\]
as $\la \to \infty$.
Inserting this in \eqref{t_2asymp}, we see
\[
	t_2 = m_{p,q} \la^{\frac{1}{qr-p+1}} - \frac{brm_{p,q}^{1-q}}{qr-p+1} \la^{\frac{1-q}{qr-p+1}} (1 + o(1)).
\]
Then $u_{2,\la} = t_2 \frac{U_p}{\| U_p \|_q}$ by \eqref{u_form},
we obtain the asymptotic formula for $u_{2,\la}(x)$.

Next, we prove the asymptotic formula for $u_{1,\la}$ when $\la >> 1$.
In this case, $t_1 \to 0$ as $\la \to \infty$ and $qr > p-1$, thus
\[
	\la \| U_p \|_q^{1-p} = g(t_1) = \frac{(t_2^q+b)^r}{t_1^{p-1}} = \frac{b^r}{t_1^{p-1}}(1 + o(1)).
\]
From this, we have
\[
	t_1 = b^{\frac{r}{p-1}} \la^{-\frac{1}{p-1}} \| U_p \|_q (1 + o(1))
\]
and we can write
\begin{equation}
\label{t_1asymp}
	t_1 = b^{\frac{r}{p-1}} \la^{-\frac{1}{p-1}} \| U_p \|_q (1 + \eta), \quad \eta = o(\la^{-\frac{1}{p-1}}).
\end{equation}
As before, we insert \eqref{t_1asymp} into
\begin{equation}
\label{Eq:t_1}
	(t_1^q + b)^r = t_1^{p-1} \la \| U_p \|_q^{1-p}.
\end{equation}
Then Taylor expansion implies
\begin{align*}
	&\text{(LHS) of \eqref{Eq:t_1}} = (t_1^q + b)^r = b^r \( 1 + \frac{r}{b} t_1^q + o(t_1^q) \). \\
	&\text{(RHS) of \eqref{Eq:t_1}} = t_1^{p-1} \la \| U_p \|_q^{1-p} \\ 
	&\overset{\eqref{t_1asymp}}{=} \la \| U_p \|_q^{1-p} \(b^{\frac{r}{p-1}} \la^{-\frac{1}{p-1}} \| U_p \|_q (1 + \eta) \)^{p-1} \\
	&= b^r (1 + \eta)^{p-1} = b^r \( 1 + (p-1) \eta + o(\eta) \). 
\end{align*}
Comparing these equations, we have
\begin{align*}
	\eta &= \frac{r}{b(p-1)} t_1^q (1 + o(1)) \\
	&\overset{\eqref{t_1asymp}}{=} \frac{r}{b(p-1)} \(b^{\frac{r}{p-1}} \la^{-\frac{1}{p-1}} \| U_p \|_q \)^q (1 + o(1)) \\
	&= \frac{r}{p-1} b^{\frac{qr-p+1}{p-1}} \la^{-\frac{q}{p-1}} \| U_p \|^q_q (1 + o(1))
\end{align*}
as $\la \to \infty$.
Returning to \eqref{t_1asymp} with this, we see
\[
	t_1 = b^{\frac{r}{p-1}} \la^{-\frac{1}{p-1}} \| U_p \|_q \(1 + \frac{r}{p-1} b^{\frac{qr-p+1}{p-1}} \la^{-\frac{q}{p-1}} \| U_p \|^q_q (1 + o(1))\).
\]
Since $u_{1,\la} = t_1 \frac{U_p}{\| U_p \|_q}$, we have the asymptotic formula for $u_{1,\la}$.
\end{proof}

\section{Proof of Theorem \ref{Thm:b>0(II)} and Theorem \ref{Thm:b>0(III)}.}

In this section, first we prove Theorem \ref{Thm:b>0(II)}.

\begin{proof}
Since $qr-p+1 \le 0$ and $b > 0$, $g(t)$ in \eqref{Def:g(t)} is strictly decreasing on $(0, +\infty)$
and
\[
	\lim_{t \to +0} g(t) = +\infty, \quad
	\lim_{t + \infty} g(t) = 
	\begin{cases}
	0, \quad qr-p+1 < 0, \\
	1, \quad qr-p+1 = 0.
	\end{cases}
\]
Thus the equation \eqref{Eq:g(t)} has the unique solution $t_{\la}$ for any $\la >0$ when $qr-p+1 <0$,
and for any $\la$ such that $\la \| U_p \|_1^{1-p }> 1$ when $qr-p+1 = 0$.
By Lemma \ref{Lemma:M} \eqref{u_form}, $u_{\la} = t_{\la} \frac{U_p}{\| U_p \|_q}$ is the unique solution of \eqref{BBP}.
This proves Theorem \ref{Thm:b>0(II)}.
\end{proof}

In some special cases, we can obtain the exact value of $t_{\la}$ in Theorem \ref{Thm:b>0(II)}.
This is the content of Theorem \ref{Thm:b>0(III)}, which we prove here.

\begin{proof}
In this case, since $q=1 < \frac{p-1}{2} = r$, the equation \eqref{Eq:g(t)} reduces to
\[
	\frac{t+b}{t^2} = \la^{\frac{2}{p-1}} \| U_p \|_1^{-2}.
\]
This is the quadratic equation in $t > 0$ and its positive solution is given by
\begin{align*}
	t_{\la} &= \frac{1 + \sqrt{1 + 4b \la^{\frac{2}{p-1}} \| U_p \|_1^{-2}}}{2 \la^{\frac{2}{p-1}} \| U_p \|_1^{-2}} \\
	&= \| U_p \|_1 \frac{\| U_p\|_1 + \sqrt{\| U_p\|_1^2 + 4b \la^{\frac{2}{p-1}}}}{2} \la^{-\frac{2}{p-1}}
\end{align*}
Thus by Lemma \ref{Lemma:M} \eqref{u_form}, we obtain the unique solution $u_{\la} = t_{\la} \frac{U_p}{\| U_p \|_1}$ of \eqref{BBP}.
\end{proof}

\section{Higher dimensional case.}

In this section, we consider the higher dimensional analogue of the problem \eqref{BBP}, namely
\begin{equation}
\label{H_BBP_M}
	\begin{cases}
	&M(\| u \|_{L^q(\Omega)}) \Delta u = u^p \quad \text{in} \ \Omega, \\
	&u > 0 \quad \text{in} \ \Omega, \\
	&u(x) \to +\infty \quad \text{as} \ d(x) := {\rm dist }(x, \pd\Omega) \to 0,
	\end{cases}
\end{equation}
where $\Omega$ is a bounded domain in $\re^N$, $N \ge 1$, $p > 1$, $q > 0$,
and $M:[0,\infty) \to \re_+$ is a continuous function.
We call $u$ a solution of \eqref{H_BBP_M} if $u \in C^2(\Omega) \cap L^q(\Omega)$ and $u$ satisfies \eqref{H_BBP_M} for $x \in \Omega$.

As in the 1-D case, the important fact is the existence and uniqueness of solutions of the model problem
\begin{equation}
\label{H_u_p}
	\begin{cases}
	&\Delta u = u^p \quad \text{in} \ \Omega, \\
	&u > 0 \quad \text{in} \ \Omega, \\
	&u(x) \to +\infty \quad \text{as} \ d(x) \to 0.
	\end{cases}
\end{equation}
It is well-known that there always exists a solution of \eqref{H_u_p} when $p >1$, 
if $\Omega$ satisfies some regularity assumption (exterior cone condition is enough) \cite{Keller} \cite{Osserman}.
A necessary and sufficient condition involving capacity for the existence of solutions of \eqref{H_u_p} is given in \cite{Labutin}.
Also the solution is unique and satisfies the estimate
\begin{equation}
\label{boundary_estimate}
	C_1 d(x)^{-\frac{2}{p-1}} \le u(x) \le C_2 d(x)^{-\frac{2}{p-1}}
\end{equation}
for some $0 < C_1 \le C_2$, if $\Omega$ is Lipschitz \cite{Bandle-Marcus(JAM)}, \cite{Bandle-Marcus(AIHP)}, \cite{LM}.
If $p \in (1, \frac{N}{N-2})$ when $N \ge 3$, or $p \in (1, +\infty)$ if $N = 1,2$,
no smoothness assumption is needed for the existence of solutions to \eqref{H_u_p} 
and the solution is unique if $\pd\Omega = \pd\ol{\Omega}$ is satisfied, see \cite{MV(AIHP)}, \cite{MV(JEE)}, \cite{Veron}, \cite{Kim}.
For any $p > 1$, \eqref{H_u_p} admits at most one solution if $\pd\Omega$ is represented as a graph of a continuous function locally \cite{MV(AIHP)}.

In the following, we assume that $\Omega$ is sufficiently smooth, say $C^2$.
In this case, the unique solution $u_p$ of the problem \eqref{H_u_p} is in $L^q(\Omega)$ for $q < \frac{p-1}{2}$.
Indeed, let $\eps > 0$ small and put  $\Omega_{\eps} = \{ x \in \Omega \ | \ d(x) > \eps \}$.
Then the estimate \eqref{boundary_estimate} implies
\begin{align*}
	\int_{\Omega} |u_p|^q dx &= \int_{\Omega_{\eps}} |u_p|^q dx +  \int_{\Omega \setminus \Omega_{\eps}} |u_p|^q dx \\
	 &\le C +  C \int_{\Omega \setminus \Omega_{\eps}} d(x)^{-\frac{2q}{p-1}}  dx \\
	 &\le C +  C' \int_0^\eps t^{-\frac{2q}{p-1}} dt < \infty
\end{align*}
if $\frac{2q}{p-1} < 1$, i.e., $q < \frac{p-1}{2}$.

As in Lemma \ref{Lemma:M}, we can prove the following.

\begin{theorem}
\label{Thm:H_M}
Let $\Omega \subset \re^N$ be a smooth bounded domain and
let $M:[0,+\infty) \to \re_+$ be continuous in \eqref{H_BBP_M}.
Then for a given $\la > 0$, the problem \eqref{H_BBP_M} has the same number of positive solutions as that of the positive solutions 
(with respect to $t >0$) of
\begin{equation}
\label{Eq:H_M}
	\frac{M(t)}{t^{p-1}} = \la \| u_p \|_q^{1-p}
\end{equation}
where $u_p$ is the unique solution of \eqref{H_u_p}.
Moreover, any solution $u_{\la}$ of \eqref{H_BBP_M} must be of the form 
\[
	u_{\la}(x) = t_{\la} \frac{u_p(x)}{\| u_p \|_q}
\]
where $t_{\la}>0$ is a solution of \eqref{Eq:H_M}.
\end{theorem}

The proof of Theorem \ref{Thm:H_M} is completely similar to that of Lemma \ref{Lemma:M}, so we omit it here.

\vskip 0.5cm

\noindent\textbf{Acknowledgement.} 
The second author (F.T.) was supported by JSPS Grant-in-Aid for Scientific Research (B), No. 23H01084, 
and was partly supported by Osaka Central University Advanced Mathematical Institute (MEXT Joint Usage/Research Center on Mathematics and Theoretical Physics).

\end{document}